\documentclass[11pt,reqno]{amsart}

\usepackage{amssymb}
\usepackage{amsmath}
\usepackage{amsthm}
\usepackage{epsfig}
\usepackage{mathrsfs}
\usepackage{graphicx}
\usepackage{xcolor}
\usepackage{bbm}
\usepackage{hyperref}
\usepackage{cleveref}
\usepackage{color}
\usepackage{float}
\usepackage[english]{babel}
\usepackage[autostyle]{csquotes}
\usepackage[margin=.8in]{geometry}
\usepackage{fullpage}
\usepackage{cite}

\newtheorem{theorem}{Theorem}[section]
\newtheorem{lemma}[theorem]{Lemma}

\newtheorem{definition}[theorem]{Definition}
\newtheorem{remark}[theorem]{Remark}
\newtheorem{corollary}[theorem]{Corollary}

\def\N{{\mathbb N}}

\def\R{{\mathbb R}}

\def\E{\mathcal E}

\def\to{\rightarrow}

\def\eps{\varepsilon}
\def\vphi{\varphi}

\def\prob{\mathcal{P}(\R^d)}
\def\ird{\int_{\R^d}}
\def\d{\,\mathrm{d}}

\newcommand{\supp}{\operatorname{supp}}
\newcommand{\measurerestr}{%
  \,\raisebox{-.127ex}{\reflectbox{\rotatebox[origin=br]{-90}{$\lnot$}}}\,%
}

\numberwithin{equation}{section}

\begin{document}
\title{Existence of ground states for\\ aggregation-diffusion equations}
\author{J. A. Carrillo}
\address{Department of Mathematics, Imperial College London, London SW7 2AZ, UK}
\email{carrillo@imperial.ac.uk}
\author{M. G. Delgadino}
\address{Department of Mathematics, Imperial College London, London SW7 2AZ, UK}
\email{m.delgadino@imperial.ac.uk}
\author{F. S. Patacchini}
\address{Deparment of Mathematical Sciences, Carnegie Mellon University, Pittsburgh, PA 15203, USA}
\email{fpatacch@math.cmu.edu}

\date{\today}
\thanks{JAC and MGD were partially supported by EPSRC grant number EP/P031587/1. FSP thanks the CNA at Carnegie Mellon University for their kind support. The authors are very grateful to the Mittag-Leffler Institute for providing a fruitful working environment during the special semester \emph{Interactions between Partial Differential Equations \& Functional Inequalities} for the period September--December 2016.}
\subjclass[2010]{}
\begin{abstract}
We analyze free energy functionals for macroscopic models of multi-agent systems interacting via pairwise attractive forces and localized repulsion. The repulsion at the level of the continuous description is modeled by pressure-related terms in the functional making it energetically favorable to spread, while the attraction is modeled through nonlocal forces. We give conditions on general entropies and interaction potentials for which neither ground states nor local minimizers exist. We show that these results are sharp for homogeneous functionals with entropies leading to degenerate diffusions while they are not sharp for fast diffusions. The particular relevant case of linear diffusion is totally clarified giving a sharp condition on the interaction potential under which the corresponding free energy functional has ground states or not.
\end{abstract}
\maketitle

\section{Introduction}

Given an \emph{interaction potential}, or \emph{kernel}, $W\colon\R^d\to (-\infty,\infty]$, an \emph{entropy}, or \emph{internal density}, \emph{function} $U\colon[0,\infty)\to \R$, and a \emph{temperature} $\eps\geq 0$, we consider the nonlinear evolution of a normalized density $\rho$, given by the equation
\begin{equation}\label{the flow}
	\partial_t \rho = \nabla\cdot\left( (\nabla W*\rho)\rho\right) + \eps \nabla\cdot\left(\nabla U'(\rho)\rho\right), \qquad t>0.
\end{equation}
This work derives conditions on the relationship of the interaction potential $W$ and the entropy function $U$ for the existence and nonexistence of stationary and ground states to \eqref{the flow}. Taking advantage of the fact that \eqref{the flow} is the $2$-Wasserstein gradient flow of the \emph{free energy}
\begin{equation}\label{the energy}
	E_\eps(\rho)=\frac{1}{2}\ird\ird W(x-y)\,d\rho(x)\,d\rho(y) + \eps \ird U(\rho(x))\,dx, 
\end{equation}
where we refer the reader for instance to \cite{Villani03,CaMcCVi03,CaMcCVi06,AGS}, we follow a strategy based on energetic arguments to show our main results. For example, the existence of stationary or ground states is obtained by analyzing suitable conditions for the free energy \eqref{the energy} to admit critical points or global minimizers, respectively. This strategy has already been successfully used to analyze general qualitative properties of local minimizers for zero temperature ($\eps=0$), as in \cite{BCLR2,CDM,CCP,simione2015existence,CFP}. 

The case of linear diffusion, which in \eqref{the flow} translates to $U(\rho)=\rho\log \rho$ and $\nabla\cdot(\nabla U'(\rho)\rho)=\Delta \rho$, is classical in the literature and corresponds to the McKean--Vlasov equation \cite{DG}. In fact, under suitable conditions on $W$, the flow \eqref{the flow} can be seen as the so-called mean-field limit of the following coupled ODE system: consider $N$ particles at positions $\{X_1,...,X_N\}\subset \R^d$ satisfying the coupled equations
\begin{equation}\label{ODE system}
\begin{cases}
\displaystyle\dot{X}_i=-\displaystyle\frac{1}{N}\sum_{j\ne i} \nabla W(X_i-X_j)+\sqrt{2\eps}B_i,\\
{X}_i(0)={X}_i^0\in\R^d,
\end{cases} \quad i=1,\dots, N\,,
\end{equation}
where $\{B_i\}_{i=1}^N$ is a family of $N$ independent Brownian processes; see \cite{szn}. The ODE system \eqref{ODE system} and its mean-field limit \eqref{the flow} are widely used in diverse applications such as granular media \cite{BCP,BCCP,Tos1}, pedestrian models \cite{CPT}, swarming \cite{mogilner1999non,TBL,KCBFL,BCCD,HF}, cell adhesion \cite{PBSG}, chemotaxis cell motility \cite{DoPe04,BlaDoPe06,BCC12}, and opinion dynamics \cite{GPY,MT,APTZ}, to name only a few. The case of homogeneous nonlinear diffusion, corresponding to $U(\rho)=\frac1{m-1}\rho^m$ and $\nabla\cdot(\nabla U'(\rho)\rho)=\Delta \rho^m$ with $m>1$, has also received a lot of attention, in particular thanks to variants of the classical Keller--Segel model for chemotaxis \cite{TBL,BCL,CCV,CHVY,calvez2017equilibria,CCHCetraro} and to swarming aggregation models taking into account size or volume constraints for the individual particles \cite{Ol,BCM,BD,FHK}.

The first main result of this paper (Theorem \ref{thm:existence-global}) gives, for interaction potentials which are bounded from below and linear diffusion, a sharp condition related to the noise strength $\eps$ allowing for the existence or not of global minimizers of the corresponding free energy \eqref{the energy}; see Section 6. It is well-known that for homogeneous kernels with linear diffusion, the criticality corresponds to the logarithmic kernel in any dimension \cite{BCC08,calvez2017equilibria}. In particular, this is the case of the celebrated classical Keller--Segel model in two dimensions \cite{DoPe04} that enjoys the critical mass phenomena. Our result is the natural counterpart for interaction potentials which are bounded from below giving the sharp constant for this dichotomy in terms of the noise strength $\eps$. More precisely, if 
$$
 \lim_{|x|\to\infty} \nabla W(x) \cdot x =L>0
$$
for some constant $L$, then there exists a critical diffusion $\eps_\mathrm{c}=L/(2d)$ separating the existence of ground states from the unboundeness from below of the free energy. Notice that this condition is satisfied for interaction potentials behaving logarithmically at infinity. This result is not covered by the concentration-compactness principle of Lions' \cite{lions1984concentration}. 

A related problem with linear diffusion concerns the question of how stable the (local) minimizers with zero diffusion $\eps=0$ are when small noise $\eps$ is switched on. This question is very much related to metastability phenomena observed in numerical simulations \cite{EK,GPY,BDZ,BCDPZ}. For instance, it is shown in \cite{CCP,simione2015existence} that for compactly supported, not $H$-stable interaction potentials, the energy \eqref{the energy} with $\eps=0$ has global minimizers given by compactly supported probability measures. We refer the reader to \cite{CCP,simione2015existence} for further considerations on $H$-stability and on the existence/nonexistence of global minimizers without noise. Typical examples include Morse and repulsive-attractive power-law potentials \cite{CHM}, as well as compactly supported repulsive-attractive potentials \cite{BDZ,GPY}. Our second main result (Theorem \ref{thm:nonexistence-local}) gives a negative answer for repulsive-attractive interaction potentials which are smooth enough showing that no critical points or stationary states of the energy \eqref{the energy} exist as soon as the linear diffusion is triggered with $\eps>0$, no matter how small the noise strength $\eps$ is. Note that this is also related to the question of finding sharp conditions for the (non)existence of steady states for kinetic systems such as the Vlasov--Poisson--Fokker--Planck system \cite{BoDo,Dolbeault}. Our strategy is based on an argument by contradiction using the nonlinear integral equation satisfied by the steady states obtained from the Euler--Lagrange conditions for the critical points; see Section 3. We conjecture that this is one of the reasons behind the metastability observed in numerical simulations \cite{EK,GPY,BDZ,BCDPZ} in this context. Note that our result asserts that no stationary state of \eqref{the flow} exists for $\eps>0$ in the whole space $\R^d$; however, for bounded domains with no-flux boundary conditions, ground states exist by compactness and lower semicontinuity of the energy. Nevertheless, the larger the domain the flatter the stationary state becomes; see Remark~\ref{rem bounded domain}. Hence, for large domains no stationary state resembles the global minimizer of \eqref{the energy} with $\eps=0$.

The third main result (Theorem \ref{thm:nonexistence-global}) contains a sufficient condition on general interaction potentials and nonlinear diffusions for the unboundeness from below of the free energy \eqref{the energy}. This result is again sharp for homogeneous diffusions $U(\rho)=\rho^m$, with $m\ge 1$, recovering previous results in \cite{BCL,calvez2017equilibria}. Moreover, by means of an explicit counterexample, we see that these conditions are not sharp for homogeneous diffusions $U(\rho)=\rho^m$, with $m\le 1$; see Section 4. Section 5 discusses the sharpness of this condition in terms of existence of global minimizers for $U(\rho)=\rho^m$, with $m> 1$, and homogeneous kernels. Although the existence of global minimizers is a consequence of Lions' concentration-compactness principle \cite{simione2015existence}, we give here an elementary different proof exploiting extra compactness properties stemming from radial decreasing rearrangements. 

In Section 2 we first collect some preliminary material necessary to the proper treatment of the free energy \eqref{the energy} in subsets of probability measures.

\section{Preliminaries}
	In this section we introduce the notation, definitions and preliminary results used throughout.
\subsection{Measures}
	We write $\prob$ the set of Borel probabilty measures and $\mathcal{P}_\mathrm{ac}(\R^d)$ the subset of $\prob$ of measures which are absolutely continuous with respect to the $d$-dimensional Lebesgue measure. Given $p\in\N$, we write $\mathcal{P}_p(\R^d)$ the subset of $\prob$ of measures with finite $p$th moment. The \emph{$p$th Wasserstein distance} $d_p(\mu,\nu)$ between two probability measures $\mu$ and $\nu$ belonging to $\mathcal{P}_p(\R^d)$ is
\begin{equation*}
	d_p(\mu,\nu) = \min_{\pi \in \Pi(\mu,\nu)} \left( \int_{\R^d\times \R^d} |x-y|^p d \pi(x,y) \right)^{1/p} ,
\end{equation*}
where $\Pi(\mu,\nu)$ is the set of tranport plans between $\mu$ and $\nu$; i.e., $\Pi(\mu,\nu)$ is the subset of $\prob \times \prob$ of measures with $\mu$ as first marginal and $\nu$ as second marginal. We also define the \emph{$\infty$-Wasserstein distance} $d_\infty(\mu,\nu)$, whenever $\mu$ and $\nu$ are compacty supported, by
\begin{equation*}
	d_\infty(\mu,\nu) = \inf_{\pi\in\Pi(\mu,\nu)} \sup_{(x,y) \in \supp\pi} |y-x|,
\end{equation*}
where the $\supp$ denotes the support.

In this paper we work with the weak-$^*$ topology of measures. We say that a sequence $(\mu_n)_{n\in\N}$ of measures in $\mathcal{M}(\R^d)$, the set of Radon measures, converges \emph{weakly-$^*$} to a measure $\mu_\infty \in \mathcal{M}(\R^d)$ if
\begin{equation*}
	\ird f(x) \,d\mu_n(x) \to \ird f(x) \,d\mu_\infty(x) \quad \mbox{as $n\to \infty$ for all $f\in C_\mathrm{c}(\R^d)$,}
\end{equation*}
where $C_\mathrm{c}(\R^d)$ is the space of compactly supported continuous functions defined on $\R^d$.

Since $\mathcal{M}(\R^d)$ is the dual space of $C_\mathrm{c}(\R^d)$, the Banach--Alaoglu theorem applied to measures tells us that the closed unit ball in $\mathcal{M}(\R^d)$ in the weak-* topology is weakly-$^*$ compact. If we now restrict to $\prob$ and to the set $\mathcal{M}_+(\R^d)$ of nonnegative Radon measures, then we get the following compactness \cite[Chapter 3]{Brezis}:
\begin{theorem}[compactness of weak-$^*$ topology]\label{weakcompact}
	Let $(\mu_n)_{n\in\N}$ be a sequence in $\prob$. There exists a subsequence of $(\mu_n)_{n\in\N}$ which converges weakly-$^*$ to some $\mu_\infty\in \mathcal{M}_+(\R^d)$. 
\end{theorem}

We recall the definition of tightness and Prokhorov's theorem; see for example \cite{Bil}.

\begin{definition}[tight family of measures]
	We say that a family $\{\mu_n\}_{n\in\N} \subset \mathcal{M}(\R^d)$ is \emph{tight} if for every $\delta > 0$ there exists a compact set $K_\delta \subset \R^d$ such that
\begin{equation*}
	|\mu_n(K_\delta^\mathrm{c})| \leq \delta, \quad \mbox{uniformly in $n$,}
\end{equation*}
where $K_\delta^\mathrm{c}$ is the complementary set of $K_\delta$.
\end{definition}
\begin{theorem}[Prokhorov's theorem]\label{prok}
	A family $\{\rho_n\}_{n\in\N} \subset \prob$ is tight if and only if it is weakly-$^*$ relatively compact in $\prob$.
\end{theorem}

Given a map $T \colon \R^d \to \R^d$ and a probability measure $\rho$, we write $T\#\rho$ the \emph{push-forward measure} of $\mu$ through $T$; i.e., $T\#\rho$ is the probability measure such that, for any measurable function $\vphi\colon \R^d \to [-\infty,\infty]$ with $\vphi\circ T$ integrable, we have
\begin{equation*}
	\ird \vphi(x) \,dT\#\rho(x) = \ird \vphi(T(x)) \d \rho(x).
\end{equation*}


\subsection{Definition of the energy}\label{subsec:defn-energy}
We introduce the family of free energy functionals defined on the set of probability measures $E_\eps\colon \prob \to (-\infty,\infty]$, indexed by $\eps>0$, given for all $\rho \in \prob$ by $E_\eps(\rho) = \mathcal{W}(\rho)+\eps\mathcal{E}_U(\rho)$,
where the interaction energy $\mathcal{W}$ and the entropy $\E_U$ are defined by
\begin{equation*}
	\mathcal{W}(\rho) = \frac12 \ird \ird W(x-y)\,d\rho(x)\,d\rho(y)\qquad \mbox{and} \qquad \mathcal{E}_U(\rho) = \ird U(\rho_\mathrm{ac}(x)) \,d x + \rho_\mathrm{s}(\R^d) U_\mathrm{s}.
\end{equation*}
Here, $W$ is the interaction potential, or kernel, and $U$ is the entropy, or internal density, function. The measures $\rho_\mathrm{ac}$ and $\rho_\mathrm{s}$ are the absolutely continuous and singular parts of $\rho$ in the unique Lebesgue decomposition $\rho=\rho_{ac}+\rho_\mathrm{s}$, and $U_\mathrm{s}$ is defined as $U_\mathrm{s} = \limsup_{r\to\infty} U(r)/r$; see \cite[Definition 2.32]{AFP} for a link to the recession function. By convention, when $U_\mathrm{s}=+\infty$, or equivalently, $U$ has superlinear growth at infinity, and $\rho_\mathrm{s}(\R^d)=0$ we set $\rho_\mathrm{s}(\R^d)U_\mathrm{s}=0$. 
Observe that when $U_\mathrm{s}=+\infty$ we get
\begin{equation*}
	\E_U(\rho) = \begin{cases} \ird U(\rho(x)) \,d x &\mbox{for all $\rho\in \mathcal{P}_\mathrm{ac}(\R^d)$},\\ +\infty & \mbox{for all $\rho \in \prob\setminus \mathcal{P}_\mathrm{ac}(\R^d)$}.  \end{cases}
\end{equation*}

Although further hypotheses may be considered in various places below, we shall most of the time, sometimes implicitly, assume that $W\colon\R^d\to (-\infty,\infty]$ is locally integrable (i.e., is in $L^1_\mathrm{loc}(\R^d)$), lower semicontinuous and symmetric (i.e., $W(x) = W(-x)$ for all $x\in\R^d$), and that $U\colon [0,\infty) \to \R$ is continuous, of class $C^2$ on $(0,\infty)$, convex, and satisfying $U(0)=0$. We refer to this set of hypotheses as {\bf (H)}. Note that, without loss of generality, we can simply assume that $W$ is positive rather than bounded from below by a real constant.

We assume that $W$ is symmetric without loss of generality, since otherwise one could symmetrize the interaction potential and the question about minimizers or critical points of these functionals, as defined below, would remain unchanged.  We shall say that the interaction potential $W$ is \emph{differentiable away from the origin} if $W$ is of class $C^1$ everywhere but $0$. 

Let us emphasize that the basic assumptions {\bf (H)} together with boundedness from below of the interaction potential $W$ ensure that the free energy functional \eqref{the energy} is well-defined on the set of probability measures. Indeed, notice first that the weak-$^*$ lower semicontinuity of the functional $\mathcal{E}_U$ is equivalent to the lower semicontinuity and convexity of $U$; see \cite[Theorem 2.34]{AFP}. Moreover, because $W$ is lower semicontinuous and bounded from below, we obtain by \cite[Proposition 7.2]{santambrogio2015optimal} that $\mathcal{W}$ is weak-$^*$ lower semicontinuous. Therefore, we obtain that, given $\{\rho_n\}_{n\in\N} \subset C_\mathrm{c}^\infty(\R^d)\cap\mathcal{P}(\R^d)$ and $\rho\in C_\mathrm{c}^\infty(\R^d)\cap\mathcal{P}(\R^d)$ such that $\rho_n\rightharpoonup\rho$, the assumption {\bf (H)} yields
\begin{equation*}
E_\eps(\rho)\le \liminf_{n\to\infty} E_\eps(\rho_n). 
\end{equation*}
Now, we extend $E_\eps$ to all of $\mathcal{P}(\R^d)$ by lower semicontinuity. Given $\rho\in\mathcal{P}(\R^d)$, we define
\begin{equation*}
E_\eps(\rho)=\inf_{\substack{\{\rho_n\}_{n\in\N}\subset C_\mathrm{c}^\infty(\R^d)\cap\mathcal{P}(\R^d)\\ \mathrm{s.t.}\;\rho_n\rightharpoonup\rho}}\liminf_{n\to\infty} E_\eps(\rho_n).
\end{equation*}
In addition, regardless of the boundedness from below of $W$, the assumptions {\bf (H)} make sure that the energy is always well-defined when restricted to characteristic functions of balls. We shall always make sure in the following to be in either of these two well-defined cases.

Each entropy function $U$ is associated its \emph{McCann scaling function} $u\colon (0,\infty) \to \R$, which is defined by
\begin{equation*}
	u(r)=r^dU(r^{-d}) \quad \mbox{for all $r\in (0,\infty)$}.
\end{equation*}
As proven by McCann \cite{McCann97}, convexity in the $2$-Wasserstein sense of the associated entropy is equivalent to $u$ being nonincreasing and convex. Observe that $u$ being nonincreasing and $U$ convex are equivalent to say that the formal $2$-Wasserstein gradient flow is a nonlinear diffusion equation of the form $\partial_t \rho=\Delta P(\rho)$ with $P$ nonnegative and nondecreasing respectively since $P(r)=rU'(r)-U(r)$ and $P'(r)=rU''(r)$. These conditions on $U$ intuitively mean that the functional is modelling a localized repulsive effect for $\rho$. We say that $U$ does \emph{not} model \emph{slow diffusion} if $\lim_{r\to0} U'(r)=-\infty$. Notice that this includes the subcase of linear diffusion $P(\rho)=\rho$, or equivalently, $U(\rho)=\rho\log \rho$, and the subcase of \emph{fast diffusion} corresponding to $\lim_{r\to0} rU''(r)=-\infty$. Note that McCann's scaling function captures how the entropy functional changes over dilations of the normalized characteristic function $\chi_{B_r}$ of the open Euclidean ball $B_r$ of radius $r\ge 0$ centered at the origin, or in other words, over rays in the Wassertein metric spaces $(\mathcal{P}_p(\R^d), d_p)$, $1\leq p \leq \infty$. More precisely, it is easy to check that
\begin{align*}
\E_U(\rho_r)=\E_U(r^{-d}\omega_d^{-1}\chi_{ B_{r}}) = \int_{B_r}U(r^{-d}\omega_d^{-1})\;dx= r^d\omega_d U(r^{-d}\omega_d^{-1})=u(r\omega_d^{1/d}),
\end{align*}
where $\rho_r=T_r \# (\omega_d^{-1}\chi_{B_1})= r^{-d}\omega_d^{-1}\chi_{ B_{r}}$ with $T_r(x)=rx$, and $\omega_d$ is the volume of the $d$-dimensional unit ball. Often we shall consider the derivative of the entropy functional under dilations, and for that purpose we consider the related \emph{scaling function} $v\colon(0,\infty) \to \R$ given by 
\begin{equation*}
	v(r)=-ru'(r) \quad \mbox{for all $r\in (0,\infty)$}.
\end{equation*}

The model cases for the diffusion are given by the power (\emph{nonlinear}) function
\begin{equation*}
	U(r) = \frac{r^m}{m-1}, \quad m\neq1, \qquad \mbox{for all $r\in [0,\infty)$},
\end{equation*}
in which case we shall write $\E_m$ instead of $\E_U$, and by the logarithmic function (which we refer to as the \emph{linear} case $m=1$)
\begin{equation*}
	U(r) = r\log r \quad \mbox{for all $r\in [0,\infty)$},
\end{equation*}
in which case we prefer the notation $\E$ over $\E_1$ for the entropy. In these typical models the associated scaling functions are given for all $r\in(0,\infty)$ by
\begin{equation*}
	u(r)=\frac{r^{(1-m)d}}{m-1} \qquad \mbox{and} \qquad v(r)=dr^{(1-m)d}
\end{equation*}
for $m\ne 1$, and by
\begin{equation*}
	u(r)=-d\log(r) \qquad \mbox{and} \qquad v(r)=d
\end{equation*}
 for $m=1$.
 
The model case for attractive interaction potentials is given by power laws. For a given $\beta>-d$ we write $W_\beta$ in place of $W$ for the interaction potential defined by
\begin{equation*}
W_\beta(x) = \begin{cases} \frac{|x|^\beta}{\beta}&\mbox{if $\beta\ne0$},\\ \log |x|&\mbox{if $\beta=0$}. \end{cases}
\end{equation*}
The resulting interaction energy in this case is denoted $\mathcal{W}_\beta$. Note that for $-d<\beta\leq 0$ one needs to restrict the functional to a set of suitable densities for the energies $\mathcal{W}_\beta$ and $\mathcal{E}_m+\mathcal{W}_\beta$ to be well-defined. For instance, these energies are always well-defined for compactly supported bounded functions. We will specify the precise definition of the domain of the energies when needed. As we shall see, there is a direct relationship between strength of attractivity in the interaction energy $\mathcal{W}_\beta$ and repulsivity in the entropy $\mathcal{E}_m$. In this paper we study in detail the criticality that happens at $\beta=0$ and $m=1$; we give sharp conditions for the existence of global minimizers in the linear diffusion regime (Theorem~\ref{thm:existence-global}). 

\subsection{Critical points, local minimizers, and Euler--Lagrange conditions}
We say that $\rho\in \mathcal{P}(\R^d)$ is a \emph{critical point} of ${E}_\eps$ if ${E}_\eps(\rho)<\infty$ and if it satisfies that $\frac{\delta {E}_\eps}{\delta \rho} :=\eps U'(\rho)+W*\rho$ be equal to a constant, possibly different in each closed connected component of the support of $\rho$. Small variants of the results in \cite{BCL,BCLR2,CDM,CCV,calvez2017equilibria} imply that local minimizers  of ${E}_\eps$ with respect to $d_p$ for any $1\leq p\leq \infty$ are critical points of ${E}_\eps$. 

For the lack of a precise reference, we derive here the Euler--Lagrange conditions for $d_\infty$-local minimizers. Given $r>0$, we say that $\rho$ is a \emph{$d_\infty$-local minimizer with radius $r$} if $E_\eps(\rho)\le E_\eps(\nu)$ for any $\nu\in\mathcal{P}(\R^d)$ such that $d_\infty(\rho,\nu)<r$. Note that this definition holds analogously for $d_p$-local minimizers for any $p\in[1,\infty)$. We show that if $\rho$ is a $d_\infty$-local minimizer with radius $r$, then it satisfies that for each closed connected component $A_i$ of its support there exists  $C_i\in \R$ such that 
\begin{equation}\label{firstorderconditionsprelim}
\begin{cases}
\eps U'(\rho)+W*\rho=C_i\qquad\mbox{almost everywhere on $A_i$},\\
\eps U'(\rho)+W*\rho\ge C_i \qquad\mbox{almost everywhere on $A_i+B_{r}$.}
\end{cases}
\end{equation}
Indeed, because $\rho$ is a $d_\infty$-local minimizer, we obtain that if $d_\infty(\nu,\rho)\le r$, then
\begin{equation*}
\left.\frac{d}{dt}E_\eps((1-t)\rho+t\nu)\right|_{t=0}\ge0,
\end{equation*}
which implies
\begin{equation}\label{E-L differential}
\ird (\eps U'(\rho)+W*\rho)\;d\nu\ge\ird (\eps U'(\rho)+W*\rho)\;d\rho. 
\end{equation}
We take $x_0\in\supp\rho$ and $\phi\in C^\infty_\mathrm{c}(B_r(x_0))$, where $B_r(x_0)$ stands for the open ball of radius $r$ and center $x_0$. For $\delta<\|\phi\|_{L^\infty(\R^d)}/2$, we consider the probability measure 
$$
\nu=\rho\measurerestr B_r^\mathrm{c}(x_0)+\left(1+\eps \left(\phi-\frac{1}{\rho(B_r(x_0))}\int_{B_r(x_0)}\phi \;d\rho\right)\right)\rho\measurerestr B_r(x_0),
$$
where $\rho\measurerestr B_r(x_0)$ denotes the restriction of $\rho$ to the ball $B_r(x_0)$ and $\rho\measurerestr B_r^\mathrm{c}(x_0)$ the restriction to its complement. Because $\nu$ results only from perturbing $\rho$ inside $B_r(x_0)$, it is clear that $d_\infty(\nu,\rho)<r$. For this particular $\nu$, \eqref{E-L differential} can be rewritten as
\begin{equation*}
\int_{B_r(x_0)} (\eps U'(\rho)+W*\rho)\phi \,d\rho\ge\left(\frac{1}{\rho(B_r(x_0))} \int_{B_r(x_0)}\phi\, d\rho
\right)\int_{B_r(x_0)} (\eps U'(\rho)+W*\rho)\,d\rho.
\end{equation*}
By taking $-\phi$ instead of $\phi$, we get that, for any $\phi\in C^\infty_\mathrm{c}(B_r(x_0))$,
\begin{equation*}
\int_{B_r(x_0)} (\eps U'(\rho)+W*\rho)\phi \,d\rho=\left(\frac{1}{\rho(B_r(x_0))} \int_{B_r(x_0)}\phi\, d\rho
\right)\int_{B_r(x_0)} (\eps U'(\rho)+W*\rho)\,d\rho,
\end{equation*}
which implies that, almost everywhere in $B_r(x_0)$,
\begin{equation*}
\eps U'(\rho)+W*\rho=\frac{1}{\rho(B_r(x_0))}\int_{B_r(x_0)} (\eps U'(\rho)+W*\rho)\;d\rho.
\end{equation*}
Hence, $\eps U'(\rho)+W*\rho$ is almost everywhere locally constant in each connected component of the support of $\rho$. Which shows the first condition in \eqref{firstorderconditionsprelim}. Next, we consider $\psi\in C^\infty_\mathrm{c}(B_r(x_0))$ positive. For $\delta<1/\ird \psi\,dx$ we now take the probability measure 
\begin{equation*}
	\nu=\rho+\delta\left(\psi\rho( B_r(x_0))-\left(\int_{B_r(x_0)}\psi \,dx\right)\rho\measurerestr B_r(x_0)\right),
\end{equation*}
which again satisfies $d_\infty(\nu,\rho)<r$. For this particular $\nu$, \eqref{E-L differential} can be rewritten as
\begin{equation*}
\rho( B_r(x_0))\int_{B_r(x_0)} (\eps U'(\rho)+W*\rho)\psi\, dx\ge\left(\int_{B_r(x_0)}\psi \,dx
\right)\int_{B_r(x_0)} (\eps U'(\rho)+W*\rho)\;d\rho.
\end{equation*}
The previous inequality holds for any $\psi\in C^\infty_\mathrm{c}(B_r(x_0))$ positive, which, combined with the first condition in \eqref{firstorderconditionsprelim}, implies the second inequality in \eqref{firstorderconditionsprelim}.


\subsection{Radial rearrangements}
We recall here Riesz's rearrangement inequality and its consequence on our energy $E_\eps$; see \cite[Chapter 3]{LieLo01}.
\begin{theorem}[Riesz's rearrangement inequality]
	Let $g,h\colon \R^d \to [0,\infty)$ be two nonnegative functions and $f\colon\R^d \to (-\infty,0]$ be a nonpositive function. Then 
\begin{equation*}
	\ird \ird f^*(x-y) g^*(y) h^*(x) \,dx\,dy \leq \ird \ird f(x-y) g(y) h(x) \,dx\,dy,
\end{equation*}
where $f^*$, $g^*$ and $h^*$ are the radially symmetric decreasing rearrangements of $f$, $g$ and $h$, respectively.
\end{theorem}

If $\rho\in\mathcal{P}_\mathrm{ac}(\R^d)$, then note that its radially symmetric decreasing rearrangement also belongs to $\mathcal{P}_\mathrm{ac}(\R^d)$; more generally, the $L^m$-norm of $\rho$ equals that of $\rho^*$. If $\beta<0$, we therefore have
$$
	\mathcal{W}_\beta(\rho^*) \leq \mathcal{W}_\beta(\rho) \qquad \mbox{and} \qquad \E_U(\rho^*) = \E_U(\rho),
$$
where the equality for the entropies follows from \cite[Section 3.3, Equation (3)]{LieLo01}. All in all we get that if $\rho \in\mathcal{P}_\mathrm{ac}(\R^d)$ and $\beta<0$, then 
\begin{equation}\label{eq:rearrangement-energy}
	E_\eps(\rho^*)\leq E_\eps(\rho).
\end{equation}
The case $\beta = 0$ is also included and satisfies $E_\eps(\rho^*)\leq E_\eps(\rho)$; see \cite[Lemma 2]{carlen1992competing}.


\subsection{Inequalities of Hardy--Littlewood--Sobolev (HLS) type}
We need some HLS-type inequalities; see \cite{BCL} and \cite[Theorem 1]{carlen1992competing}.

\begin{theorem}[variation of the HLS inequality]\label{variationofHLS}
	Given $\rho\in \mathcal{P}_\mathrm{ac}\cap L^{m}(\R^d)$ and $-d<\lambda<0$, for any $m\ge1-\lambda/d$ 
\begin{equation*}
	\ird\ird |x-y|^{\lambda}\,d\rho(x)\,d\rho(y)\le C(\lambda,d) \|\rho\|_{L^{m}(\R^d)}^{(1-\theta)m_\mathrm{c}},
\end{equation*}
where $C(\lambda,d,m)>0$ is a constant depending on $\lambda$, $d$ and $m$, $0\leq \theta < 1$ and $m_\mathrm{c}=1-\lambda/d$.
\end{theorem}
\begin{proof}
	From \cite[Theorem 3.1]{calvez2017equilibria} we know that for $m_\mathrm{c}=1-\lambda/d$.
\begin{equation}\label{HLSfromCCH}
	\ird\ird |x-y|^{\lambda}d\rho(x)d\rho(y)\le C(\lambda,d) \int_{\R^d}\rho(x)^{m_\mathrm{c}}\,dx.
\end{equation}
By interpolation, taking $\theta$ satisfying
\begin{equation*}
	\frac{1}{m_\mathrm{c}}=\theta+\frac{1-\theta}{m}
\end{equation*}
gives the inequality
\begin{equation}\label{interpolation}
	\|\rho\|^{m_\mathrm{c}}_{L^{m_\mathrm{c}}(\R^d)}\le \|\rho\|_{1}^{\theta m_\mathrm{c}}\|\rho\|_{L^{m}(\R^d)}^{(1-\theta)m_\mathrm{c}}.
\end{equation}
Using \eqref{HLSfromCCH}, \eqref{interpolation} and the fact that $\|\rho\|_{L^1}=1$, one can derive the desired inequality
\begin{equation*}
	\ird\ird |x-y|^{\lambda}d\rho(x)d\rho(y)\le C(\lambda,d) \|\rho\|^{m_\mathrm{c}}_{L^{m_\mathrm{c}}(\R^d)}\le C(\lambda,d) \|\rho\|_{1}^{\theta m_\mathrm{c}}\|\rho\|_{L^{m}(\R^d)}^{(1-\theta)m_\mathrm{c}}.\qedhere
\end{equation*}
\end{proof}

\begin{theorem}[logarithmic HLS inequality] \label{thm: log HLS}
Let $\rho\in\mathcal{P}_\mathrm{ac}(\R^d)$ satisfy $\log(1+|\cdot|^2)\rho \in L^{1}(\R^d)$. Then there exists $C_0\in \R$ depending only on $d$, such that
\begin{equation}\label{logHLS}
	-\ird \ird \log(|x-y|)\rho(x)\rho(y)\,dx\,dy\le \frac{1}{d}\mathcal{E}(\rho)+C_0.
\end{equation}
\end{theorem}


\subsection{Compactness of probability measures with bounded interaction energy}
We first prove the following lemma, which we shall use throughout the paper.
\begin{lemma}\label{lem:alphamomentsinequality}
	Given $\alpha>0$, there exists $\gamma_\alpha>0$ such that, for any $\rho\in\prob\cap C_\mathrm{c}^\infty(\R^d)$ satisfying $\ird x\,d\rho(x)=0$, we have
\begin{equation*}
2\max\{1,2^{\alpha-1}\} \ird |x|^\alpha \,d\rho(x) \ge \ird\ird |x-y|^\alpha\,d\rho(x)\,d\rho(y)\ge\gamma_\alpha \ird |x|^\alpha\,d\rho(x).
\end{equation*}
\end{lemma}

To prove Lemma~\ref{lem:alphamomentsinequality} we need a variation of the classical triangle inequality.
\begin{lemma}\label{aux:triangleinequality} 
Given $\alpha>0$ and $x,y\in\R^d$, we have the inequality
\begin{equation}\label{trianglealpha}
	|x-y|^\alpha\le \max\{1,2^{\alpha-1}\}(|x|^\alpha+|y|^\alpha).
\end{equation}
\end{lemma}
\begin{proof}
Suppose first $\alpha\geq1$. In this case we exploit the monotonicity and the convexity of the function $t \mapsto t^\alpha$ and the triangle inequality to say
\begin{equation*}
	\frac{|x-y|^\alpha}{2^\alpha}=\left|\frac{x-y}{2}\right|^\alpha\le\left(\frac{|x|+|y|}{2}\right)^\alpha \le \frac{1}{2}(|x|^\alpha+|y|^\alpha).
\end{equation*}
Multiplying the previous equation by $2^\alpha$ gives \eqref{trianglealpha} for $\alpha\ge 1$.

We now suppose $\alpha<1$. In this case we exploit the sublinear growth of the function $t \mapsto t^\alpha$, namely that, for any $r\geq0$,
\begin{equation*}
(1+r)^\alpha\le 1+ r^\alpha.
\end{equation*}
Assuming $x\neq0$, or the result is trivial, the triangle inequality and the previous inequality yield
\begin{equation*}
	|x-y|^\alpha\le (|x|+|y|)^\alpha=|x|^\alpha\left(1+\frac{|y|}{|x|}\right)^\alpha\le|x|^\alpha\left(1+\frac{|y|^\alpha}{|x|^\alpha}\right)=|x|^\alpha+|y|^\alpha,
\end{equation*}
which gives the desired result.
\end{proof}

\begin{proof}[Proof of Lemma \ref{lem:alphamomentsinequality}]
	The upper bound follows easily from Lemma~\ref{aux:triangleinequality}. We show the lower bound by classical compactness arguments. By contradiction we suppose that the inequality does not hold for any $\gamma_\alpha>0$. Therefore, we can assume that there exists a sequence $(\rho_n)_{n\in\N}\subset\mathcal{P}(\R^d)\cap C_\mathrm{c}^\infty(\R^d)$ such that, for all $n \in\N$,
\begin{equation*}
	\ird\ird |x-y|^\alpha\,d\rho_n(x)\,d\rho_n(y)<\frac{1}{n} \ird |x|^\alpha\,d\rho_n(x).
\end{equation*}
Next, we realize that the same inequality is satisfied by any arbitrary rescaling of our sequence $(T_{r_n}\#\rho_n)_{n\in\N}$. This follows from the scalings
\begin{equation*}
	\ird |x|^\alpha\,dT_r\#\rho_n(x)=r^\alpha\ird |x|^\alpha\,d\rho_n(x)
\end{equation*}
and
\begin{equation*}
	\ird\ird |x-y|^\alpha\,dT_r\#\rho_n(x)\,dT_r\#\rho_n(y)=r^\alpha\ird\ird |x-y|^\alpha\,d\rho_n(x)\,d\rho_n(y).
\end{equation*}
Then, for any $n$, we pick $r_n>0$ such that
\begin{equation*}
	\ird |x|^\alpha\,dT_{r_n}\#\rho_n(x)=1.
\end{equation*}
Using this we define a sequence $(\nu_n)_{n\in\N}$ of probability measures such that, for all $n\in\N$, $\nu_n=T_{r_n}\#\rho_n$. This satisfies
\begin{equation}\label{eq:alphamoment}
	\ird |x|^\alpha\,d\nu_n(x)=1\quad \mbox{and}\quad \ird\ird |x-y|^\alpha\,d\nu_n(x)\,d\nu_n(y)<\frac{1}{n}\quad\mbox{for any $n\in\N$.}
\end{equation}
By \eqref{eq:alphamoment} we know that $\{\nu_n\}_{n\in\N}$ is a tight family due to Theorem \ref{prok}, since 
$$
\int_{B_R^\mathrm{c}} d\nu_n(x)\leq \frac1{R^\alpha}\int_{B_R^\mathrm{c}} |x|^\alpha\,d\nu_n(x)=\frac1{R^\alpha}\,.
$$
Therefore there exists $\nu_\infty\in\mathcal{P}(\R^d)$, such that $\nu_n\rightharpoonup\nu_\infty$ weakly-$^*$ as $n\to\infty$. By lower semicontinuity of the interaction energy and \eqref{eq:alphamoment} we get that
\begin{equation*}
\ird\ird |x-y|^\alpha\,d\nu_\infty(x)\,d\nu_\infty(y)=0.
\end{equation*}
Using the hypothesis that $\ird x\,d\nu_n(x)=0$, we deduce that $\nu_\infty=\delta_0$. Now, we derive a contradiction from this; in particular we want to show that 
\begin{equation}\label{lowerbound}
	\ird\ird |x-y|^\alpha \,d\nu_n(x)\,d\nu_n(y)>\frac{1}{\max\{2,2^{\alpha}\}} \qquad\mbox{for $n$ large enough},
\end{equation}
which contradicts \eqref{eq:alphamoment}. Because $\nu_n\rightharpoonup\delta_0$, for any $\eps>0$ and $\eta>0$ there exists $n_0\in\N$ large enough such that
\begin{equation*}
\nu_n(B_\eta)\ge 1-\eps \qquad\mbox{for any $n\ge n_0$.}
\end{equation*}
Then, we deduce
\begin{align*}
	\ird\ird |x-y|^\alpha\;d\nu_n(x)d\nu_n(y) &\displaystyle\ge \int_{B_\eta^\mathrm{c}}\int_{B_\eta} |x-y|^\alpha\;d\nu_n(x)d\nu_n(y) \ge (1-\eps)\int_{B_\eta^\mathrm{c}}(|y|-\eta)^\alpha\;d\nu_n(y)\\
	&\ge \frac{1-\eps}{\max\{1,2^{\alpha-1}\}}\int_{B_\eta^\mathrm{c}}|y|^\alpha\;d\nu_n(y)-(1-\eps)\eta^\alpha\\
	&= \frac{1-\eps}{\max\{1,2^{\alpha-1}\}}\left(1-\int_{B_\eta}|y|^\alpha\;d\nu_n(y)\right)-(1-\eps)\eta^\alpha\\
	&\ge \frac{1-\eps}{\max\{1,2^{\alpha-1}\}}(1-\eta^\alpha)-(1-\eps)\eta^\alpha,
\end{align*}
where we have used the triangle inequality $\max\{1,2^{\alpha-1}\}((|y|-\eta)^\alpha+\eta^\alpha)\ge |y|^\alpha$ given by Lemma~\ref{aux:triangleinequality}. Taking $\eta$ and $\eps$ small enough yields \eqref{lowerbound}, which in turn shows the desired contradiction.
\end{proof}

Because sequences of probability measures with uniformly bounded moments must be tight, from Lemma~\ref{lem:alphamomentsinequality} we can easily deduce the following: given a sequence of $(\rho_n)_{n\in\N}\subset\prob$ with center of mass at zero, if 
\begin{equation*}
	\sup_{n\in\N}\ird\ird |x-y|^\alpha \,d\rho_n(x)\,d\rho_n(y)<\infty,
\end{equation*}
then $\{\rho_n\}_{n\in\N}$ is a tight family of probability measures, implying that $\{\rho_n\}_{n\in\N}$ is weakly-$^*$ relatively compact. In this work, we make use of the following more refined version of this observation, whose proof does not directly require Lemma~\ref{lem:alphamomentsinequality}

\begin{lemma}\label{lem tightness through the interaction energy}
Let us consider a positive interaction potential satisfying {\bf (H)} such that
\begin{equation}\label{hyp growth of W}
	\lim_{|x|\to\infty}W(x)=+\infty.
\end{equation}
Given a sequence $(\rho_n)_{n\in\N} \subset \mathcal{P}(\R^d)$, if 
\begin{equation}\label{boundednessoftheinteraction}
	\liminf_{n\to\infty}\ird \ird W(x-y)\;d\rho_n(x)d\rho_n(y)<\infty,
\end{equation}
then $\{\rho_n\}_{n\in\N}$ is weakly-$^*$ relatively compact up to translations. That is to say, there exists $\rho_\infty\in\prob$, a subsequence $\{\rho_{n_i}\}_{i\in\N}$ and a sequence of points $\{y_i\}_{i\in\N} \subset \R^d$, such that 
\begin{equation*}
H_{y_i}\#\rho_{n_i}\rightharpoonup \rho_\infty \quad \mbox{for all $i \in\N$},
\end{equation*}
where $H_{y_i}\colon \R^d\to \R^d$ is the translation map $H_{y_i}(x)=x-y_i$.
\end{lemma}

\begin{proof}
We show the result by contradiction. Assume that $\{\rho_n\}_{n\in\N}$ is not weakly-$^*$ relatively compact up to translations. By Prokhorov's theorem (Theorem \ref{prok}), this means that $\{\rho_n\}_{n\in\N}$ is not tight up to translations. That is, there exists $\eps>0$ such that, for any $R>0$,
\begin{equation}\label{nottight}
\liminf_{n\to\infty}\inf_{x\in\R^d} \rho_n(B_R^\mathrm{c}(x))>\eps,
\end{equation}
where $B_R(x)$ is the ball of radius $R$ centered at $x$.
Given $R>0$, we estimate the interaction energy as
\begin{align*}
	\ird \ird W(x-y)\;d\rho_n(x)d\rho_n(y) &\ge \int_{\{x,y \in \R^d \mid |x-y|>R\}} W(x-y)\;d\rho_n(x)d\rho_n(y)\\
	&\ge \inf_{z\in B_R^\mathrm{c}}W(z) \inf_{x\in\R^d} \rho_n(B_R^\mathrm{c}(x)).
\end{align*}
Using \eqref{nottight} and taking the limit when $n\to\infty$ we have
\begin{equation*}
\liminf_{n\to\infty} \ird \ird W(x-y)\;d\rho_n(x)d\rho_n(y)\ge \eps \inf_{z\in B_R^\mathrm{c}}W(z). 
\end{equation*}
By taking the limit $R\to\infty$ and using our hypothesis in \eqref{hyp growth of W} we have 
\begin{equation*}
\liminf_{n\to\infty} \ird \ird W(x-y)\;d\rho_n(x)d\rho_n(y)=\infty,
\end{equation*}
which contradicts the boundedness of the interaction energy assumed in \eqref{boundednessoftheinteraction}.
\end{proof}

\begin{remark}
Lemma~\ref{lem tightness through the interaction energy} tacitly appears in \cite{simione2015existence}. Here we give a simple alternative proof that only uses the classical Prokhorov's theorem and does not employ the more refined concentration-compactness lemma of \cite{lions1984concentration}.
\end{remark}

Finally in this section we give a quick corollary of Lemma~\ref{lem:alphamomentsinequality}: we can bound the logarithmic entropy $\E$ by the interaction energy $\mathcal{W}_\beta$.
\begin{corollary}\label{cor: Carleman ineq}
For any $\beta,\eps>0$ there exists $C_{\beta,\eps}>0$ such that, for any $\rho\in\prob$, we have that
\begin{equation*}
\mathcal{E}(\rho)\ge -C_{\beta,\eps}-\eps \mathcal{W}_\beta(\rho).
\end{equation*}
\end{corollary}
\begin{proof}
	Choose $\beta,\eps>0$. We recall the classical version of Carleman's inequality \cite[Lemma 3.4]{BCC12}: for any $\beta,\eps_0>0$ there exists $C_{\beta,\eps_0}>0$ such that, for any $\rho\in\prob$, we have
\begin{equation*}
\mathcal{E}(\rho)\ge -C_{\beta,\eps_0}-\eps_0 \ird |x|^\beta \,d\rho(x).
\end{equation*}
Using Lemma~\ref{lem:alphamomentsinequality} and picking $2\beta\eps_0=\gamma_\beta\eps$, we obtain the desired result.
\end{proof}


\section{Nonexistence of local minimizers and critical points}
\begin{theorem}[nonexistence of local minimizers and critical points]\label{thm:nonexistence-local}
Let $U(r)=r\log(r)$, i.e., we consider the case of linear diffusion. Suppose that the interaction potential $W$ is positive, satisfies {\bf (H)} and that, for any $\delta>0$, $W \in L^\infty(\R^d\setminus B_\delta)$. Then for any $\eps>0$ the energy $E_\eps=\eps \mathcal{E}+\mathcal{W}$ does not admit any $d_p$-local minimizer for any $p\in [1,\infty]$ in $\mathcal{P}(\R^d)$.  Moreover, if $W$ is Lipschitz continuous, then there are no critical points of $E_\eps$ in $\mathcal{P}_\mathrm{ac}(\R^d)$.
\end{theorem}

\begin{remark}\label{rem bounded domain}
Global minimizers of $E_\eps=\eps \mathcal{E}+\mathcal{W}$ always exist in a bounded domain $\Omega$ with null flux boundary conditions; this follows from compactness and the lower semicontinuity of the energy. If $W$ is Lipschitz, the argument below shows that, for any steady state $\rho$,
\begin{equation}
\|\rho\|_{L^\infty(\R^d)}\le |\Omega|^{-1}e^{\frac{\|W\|_{L^\infty}-\inf W}{\eps}}.
\end{equation}
Hence, at fixed $\eps>0$ the larger the domain, the smaller the $L^\infty$-norm of any steady state is.
\end{remark}

\begin{remark}
Following a similar proof as below, Theorem~{\rm\ref{thm:nonexistence-local}} extends to any entropy functional $\mathcal{E}_U$, where $U$ is convex with $u$ nonincreasing and with $\lim_{r\to 0}U'(r)=-\infty$ (that is, $U$ does not model slow diffusion). We can also observe from the proof that the Lipschitz hypothesis on $W$ can be relaxed if we assume extra integrability on the critical points.
\end{remark}

\begin{proof}[Proof of Theorem {\rm\ref{thm:nonexistence-local}}]
We show that a minimizer does not exist by contradicting the mass condition. Assume there exists $\rho\in\mathcal{P}(\R^d)$ which is a $d_p$-local minimizer of $E_\eps$. The minimality condition \eqref{firstorderconditionsprelim} implies for each closed connected component $A_i$ of the support of $\rho$, there exists  $C_i\in \R$ and an open neighborhood $N_i$ of $A_i$ such that 
\begin{equation}\label{firstorderconditions}
\begin{cases}
\eps \log(\rho)+W*\rho=C_i\qquad\mbox{on $A_i$},\\
\eps \log(\rho)+W*\rho\ge C_i \qquad\mbox{on $N_i$.}
\end{cases}
\end{equation}

We first want to show that $A_1=N_1$, which implies that $A_1=\R^d$ as it is nonempty, open and closed. We suppose that there exists $x\in N_1\setminus \supp\rho$. By our boundedness assumption on $W$ away from the origin, we get
\begin{equation*}
|W*\rho(x)|\le \|W\|_{L^{\infty}(\R^d\setminus B_{d(x,\supp\rho)})},
\end{equation*}
where $d(x,\supp\rho)$ stands for the distance between the point $x$ and the set $\supp \rho$. Using the second equation in \eqref{firstorderconditions}, we get the contradiction
\begin{equation*}
 -\infty=\eps \log(0)=\eps \log(\rho(x))\ge C_i -\|W\|_{L^{\infty}(\R^d\setminus B_{d(x,\supp\rho)})}>-\infty.
\end{equation*}
Therefore, we have $A_1=N_1$ and so $A_1=\R^d$.
This implies that, for any $x\in \R^d$,
\begin{equation}\label{formulaforrho}
\rho(x)=\exp{\left(\frac{C_1-W*\rho(x)}{\eps}\right)}.
\end{equation}
From this equation, and using the monotonicity of the exponential, we have the bound
\begin{equation*}
\|\rho\|_{L^\infty(\R^d)}\le \exp{\left(\frac{C_1-\inf W}{\eps}\right)}.
\end{equation*}
Combining this with the local integrability and boundedness hypothesis on $W$, we get that
\begin{align*}
	\|W*\rho\|_{L^\infty(\R^d)} &\le \|W\chi_{B_\delta}*\rho\|_{L^\infty(\R^d)} + \|W\chi_{B_\delta^\mathrm{c}}*\rho\|_{L^\infty(\R^d)}\\
	&\le \|W\|_{L^1(B_\delta)}\|\rho\|_{L^\infty(\R^d)}+\|W\|_{L^\infty(B_\delta^\mathrm{c})}\|\rho\|_{L^1(\R^d)}.
\end{align*}
Using \eqref{formulaforrho}, the monotonicity of the exponential and the above inequality, we finally deduce
\begin{equation*}
\rho(x)\ge \exp\left(\frac{C_1-\|W*\rho\|_{L^\infty(\R^d)}}{\eps}\right)>0 \quad \mbox{for all $x\in\R^d$},
\end{equation*}
which clearly contradicts
\begin{equation*}
\ird \rho(x)\;dx=1,
\end{equation*}
and thus shows the nonexistence of local minimizers.

Now, we assume that $W$ is Lipschitz continuous and show that this implies the nonexistence of critical points. If $\rho\in\mathcal{P}_\mathrm{ac}(\R^d)$ is a critical point, then for each connected component $A_i$ of the support of $\rho$ there exists $C_i\in \R$ such that 
\begin{equation}\label{critical point ci}
\eps \log(\rho)+W*\rho=C_i\qquad\mbox{on $A_i$}.
\end{equation}
We also get that, in the sense of distributions,
\begin{equation}\label{critical point equation}
\nabla\rho=\rho\nabla W*\rho.
\end{equation}
Using that $W$ is Lipschitz, we get that $\nabla \rho\in L^1(\R^d)$. By the Sobolev inequality, this implies that $\rho\in L^{d/(d-1)}(\R^d)$. By using \eqref{critical point equation} and the Lipschitz condition again, we obtain that $\nabla\rho\in L^{d/(d-1)}$. Iterating we get that there exists $\alpha>0$ such that $\rho\in C^\alpha(\R^d)$. (See \cite{CCV,CHVY}, where similar arguments are applied to nonlinear fast diffusion.) Combining the smoothness of $\rho$ with \eqref{critical point ci}, we get that $\rho$ cannot vanish. Therefore, there exists $C\in\R$ such that
\begin{equation*}
\eps \log(\rho)+W*\rho=C \qquad\mbox{in $\R^d$},
\end{equation*}
and the contradiction follows as in the case above of local minimizers.
\end{proof}


\section{Nonexistence of minimizers}
We start by showing a general theorem with respect to the nonexistence of global minimizers.
\begin{theorem}[nonexistence of global minimizers] \label{thm:nonexistence-global}
Let $\eps>0$ and suppose that the interaction potential $W$ satisfies {\bf (H)} and is differentiable away from the origin, and suppose that $U$ is such that $u$ is nonincreasing. If
\begin{equation}\label{hyp:nonexistence}
	\limsup_{r\to\infty} \left( \frac{1}{2}\sup_{z\in B_{2r}} \left( \nabla W(z)\cdot z \right) -\eps v(r\omega_d^{1/d}) \right)<0
\end{equation}
or
\begin{equation}\label{hyp:nonexistence2}
\liminf_{r\to0} \left( \frac{1}{2}\inf_{z\in B_{2r}} \left( \nabla W(z)\cdot z \right) -\eps v(r\omega_d^{1/d}) \right)>0,
\end{equation}
then $E_\eps$ is not bounded below in the class of compactly supported bounded functions.
\end{theorem}

\begin{remark}\label{homogeneous non existence}
If we consider the model cases $U=U_m$ and $\mathcal{W}=\mathcal{W}_{\beta}$, then the hypotheses in \eqref{hyp:nonexistence} and \eqref{hyp:nonexistence2} translate, respectively, to
\begin{equation*}
	\lim_{r\to\infty} \left( 2^{\beta-1}r^\beta-\eps d \omega_d^{1-m} r^{(1-m)d} \right)<0 \quad \mbox{if $\beta\geq0$},
\end{equation*}
and 
\begin{equation*}
\lim_{r\to0} \left( 2^{\beta-1}r^\beta-\eps d \omega_d^{1-m} r^{(1-m)d} \right)>0 \quad \mbox{if $\beta\leq0$}.
\end{equation*}
Therefore, Theorem~\ref{thm:nonexistence-global} shows that the energy is not bounded below in the so-called aggregation-dominated \cite{calvez2017equilibria,CCHCetraro} regime 
\begin{equation*}
	\beta< d(1-m).
\end{equation*} 
For the critical case $\beta=d(1-m)$, the theorem tells us that it depends on the size of $\eps$. Indeed, the energy is not bounded below in the class of compactly supported bounded functions in the case
\begin{equation*}
	m<1 \quad \mbox{and} \quad \eps> \frac{2^{\beta-1}}{d \omega_d^{1-m}},
\end{equation*}
and in the case 
\begin{equation*}
	m>1 \quad \mbox{and} \quad \eps< \frac{2^{\beta-1}}{d \omega_d^{1-m}}. 
\end{equation*}
We show that the relationship $\beta=d(1-m)$ is sharp in the cases $m>1$ and $m=1$, and that it is not sharp in the case $m<1$; see Theorem~{\rm\ref{thm:existencem>1}}, Theorem~{\rm\ref{thm:existence-global}} and Theorem~{\rm\ref{thm:nonsharpnessfastdiffussion}}, respectively. The fair competition regime corresponding to $\beta=d(1-m)$, with $m\geq 1$, enjoys a critical mass dichotomy based on variants of the HLS inequality for $m>1$, Theorem {\rm\ref{variationofHLS}}, and the logarithmic HLS inequality for $m=1$ (Theorem {\rm\ref{thm: log HLS}}); see \cite{DoPe04,BlaDoPe06,calvez2017equilibria} for details.
\end{remark}

\begin{proof}[Proof of Theorem {\rm\ref{thm:nonexistence-global}}]
Considering $\rho_r=r^{-d}\omega_d^{-1}\chi_{ B_{r}}$, we claim that if the hypothesis in \eqref{hyp:nonexistence} is satisfied, then $\lim_{r\to\infty}E_\eps(\rho_r)=-\infty$. On the other hand, if \eqref{hyp:nonexistence2} is satisfied, then $\lim_{r\to0}E_\eps(\rho_r)=-\infty$. The proof is based on differentiating the energy under the scaling parameter $r$.
Changing variables, we obtain that
\begin{align*}
\displaystyle E_\eps(\rho_r)=\frac{1}{2\omega_d^{2}}\int_{B_1}\int_{B_1}W(r(x-y))\;dxdy+\eps r^d\omega^d U(r^{-d}\omega_d^{-1}).
\end{align*}
We remind that we denote $r^d\omega^d U(r^{-d}\omega_d^{-1})=u(r\omega_d^{1/d})$, where $u$ is the McCann scaling function. Differentiating in $r$ we obtain that
\begin{equation*}
\frac{dE_\eps(\rho_r)}{dr}=\frac{1}{r}\left(\frac{1}{2\omega_d^{2}}\int_{B_1}\int_{B_1}\nabla W(r(x-y))\cdot r(x-y)\;dxdy-\eps v(r\omega_d^{1/d})\right).
\end{equation*}
Estimating the integral we get that
\begin{equation*}
\frac{1}{r}\left(\frac{1}{2}\inf_{z\in B_{2r}}\nabla W(z)\cdot z-\eps v(r\omega_d^{1/d})\right)\le \frac{dE_\eps(\rho_r)}{dr}\le\frac{1}{r}\left(\frac{1}{2}\sup_{z\in B_{2r}}\nabla W(z)\cdot z-\eps v(r\omega_d^{1/d})\right)
\end{equation*}

If \eqref{hyp:nonexistence} is satisfied, there exists $r_1>0$ and $\delta_1>0$, such that
\begin{equation*}
\frac{dE_\eps(\rho_r)}{dr}\le -\frac{\delta_1}{r}\qquad\mbox{for any $r>r_1$}.
\end{equation*}
Integrating we get that for any $r>r_1$
\begin{equation*}
E_\eps(\rho_{r})\le \delta_1 \log(r_1/r)+ E_\eps(\rho_{r_1}).
\end{equation*}
Hence,
\begin{equation*}
\lim_{r\to\infty} E_\eps(\rho_{r})=-\infty.
\end{equation*}

If now \eqref{hyp:nonexistence2} is satisfied, there exists $r_2>0$ and $\delta_2>0$ such that
\begin{equation*}
\frac{dE_\eps(\rho_r)}{dr}\ge \frac{\delta_2}{r}\qquad\mbox{for any $r<r_2$}.
\end{equation*}
Integrating we get that for any $r<r_2$
\begin{equation*}
E_\eps(\rho_{r})\le \delta_2 \log(r/r_2)+ E_\eps(\rho_{r_2}).
\end{equation*}
Hence,
\begin{equation*}
\lim_{r\to0} E_\eps(\rho_{r})=-\infty. \qedhere
\end{equation*}
\end{proof}

For homogeneous energies, we can show that Theorem~\ref{thm:nonexistence-global} is not sharp in the fast diffusion case. 

\begin{theorem}[unboundedness of energy for fast diffusion $m<1$]\label{thm:nonsharpnessfastdiffussion}
Given $m<1$ and $0<\beta<d(1-m)/m$, then for any $\eps>0$ there exists $\rho\in\mathcal{P}(\R^d)\cap L^\infty(\R^d)$ such that
\begin{equation*}
E_\eps(\rho)=\int_{\R^d}\int_{\R^d} \frac{|x-y|^{\beta}}{\beta}\;d\rho(x)d\rho(y)+\eps \int_{\R^d}\frac{\rho^m(x)}{m-1}\;dx=-\infty.
\end{equation*}
\end{theorem}
\begin{proof}
We construct a probability measure such that the entropy functional $\mathcal{E}_m$ is infinite but the interaction energy $\mathcal{W}_\beta$ is bounded.
Decomposing $\R^d$ into dyadic rings, we consider
\begin{equation}\label{dyadicrho}
\rho=\sum_{k=0}^\infty \frac{\rho_k}{|B_{2^{k+1}}\setminus B_{2^{k}}|}\chi_{B_{2^{k+1}}\setminus B_{2^{k}}},
\end{equation}
where
\begin{equation*}
\rho_k=\frac{2^{-k\gamma}}{\sum_{\ell=0}^\infty 2^{-\ell\gamma}}.
\end{equation*}
Now we want to pick $\gamma>0$ appropriately, such that 
\begin{equation*}
\int_{\R^d}\int_{\R^d} \frac{|x-y|^{\beta}}{\beta}\;d\rho(x)d\rho(y)<\infty
\end{equation*}
and
\begin{equation*}
\int_{\R^d}\frac{\rho^m(x)}{m-1}\;dx=-\infty.
\end{equation*}
By Lemma~\ref{aux:triangleinequality} we know $|x-y|^{\beta}\le\max\{1,2^{\beta-1}\}(|x|^\beta+|y|^\beta)$ for all $x,y \in \R^d$; hence
\begin{equation*}
\int_{\R^d}\int_{\R^d} \frac{|x-y|^{\beta}}{\beta}\;d\rho(x)d\rho(y)\le \frac{2\max\{1,2^{\beta-1}\}}{\beta}\ird |x|^\beta\;d\rho(x).
\end{equation*}
Using \eqref{dyadicrho}, the exact form for $\rho$, we get
\begin{equation*}
\ird |x|^\beta\;d\rho=\sum_{k=0}^\infty \frac{\rho_k}{|B_{2^{k+1}}\setminus B_{2^{k}}|}\int_{B_{2^{k+1}}\setminus B_{2^{k}}}|x|^\beta\;dx= C_1(d,\beta)\sum_{k=0}^\infty \rho_k 2^{k\beta}=C_1(d,\beta) \frac{\sum_{k=0}^\infty 2^{-k(\gamma-\beta)}}{\sum_{\ell=0}^\infty 2^{-\ell\gamma}},
\end{equation*}
where $C_1(d,\beta)$ is a constant that depends only on the dimension $d$ and $\beta$. In order for the right hand side to be finite, which in turn bounds the interaction energy, we need 
\begin{equation}\label{cond1}
\gamma>\beta.
\end{equation}
Turning our attention to the entropy and using the exact form of $\rho$ again, we get
\begin{equation*}
\ird \rho^m(x)\;dx=\sum_{k=0}^\infty \rho_k^m|B_{2^{k+1}}\setminus B_{2^{k}}|^{(1-m)}=C_2(d,\beta)\frac{\sum_{k=0}^\infty 2^{-k(m\gamma-d(1-m))}}{\left(\sum_{\ell=0}^\infty 2^{-\ell\gamma}\right)^m},
\end{equation*}
for some constant $C_2(d,\beta)$. For the right hand side to be infinite, we need
\begin{equation}\label{cond2}
m\gamma<d(1-m).
\end{equation}
Therefore, combining \eqref{cond1} and \eqref{cond2}, we get
\begin{equation*}
\beta<\gamma<\frac{d(1-m)}{m}.
\end{equation*}
By hypothesis, we have that $\beta<d(1-m)/m$, which implies we can take any $\gamma$ in between.
\end{proof}

\section{Existence of minimizers for $m>1$}
Again considering homogeneous energies, we now show that Theorem~\ref{thm:nonexistence-global} is sharp when $m>1$.
\begin{theorem}[existence of global minimizers for slow diffusion $m>1$]\label{thm:existencem>1}
Given $m>1$ and $0>\beta>(1-m)d$, for any $\eps>0$ the energy $E_\eps$, given by
\begin{equation*}
	E_\eps(\rho) = \begin{cases} \displaystyle \ird\ird \frac{|x-y|^\beta}{\beta} \,d\rho(x)\,d\rho(y) + \eps \ird \frac{\rho^m(x)}{m-1}\, dx & \mbox{for all $\rho \in L^m(\R^d)$},\\[5mm] +\infty & \mbox{for all $\rho \in \prob \setminus L^m(\R^d)$}, \end{cases}
\end{equation*} 
has a global minimizer.
\end{theorem}
\begin{proof}
We show the result by the direct method of the calculus of variations. We first observe that the functional $E_\eps(\rho)$ is well-defined and bounded below by the variant of the HLS inequality in Theorem \ref{variationofHLS}.

\medskip

\textit{Step 1.} We show that there exists $(\mu_n)_{n\in\N}\subset\mathcal{P}(\R^d)\cap L^m(\R^d)$ a minimizing sequence such that
\begin{equation}\label{propertiesofmun}
\lim_{n\to\infty}E_\eps(\mu_n)=\inf E_\eps>-\infty\quad\mbox{and}\quad\mathcal{W}_\beta (\mu_n)= \eps \frac{d(m-1)}{\beta}\mathcal{E}_m(\mu_n) \quad\forall n\in\N.
\end{equation}
Moreover, each $\mu_n$ is radially symmetric and decreasing and $E(\mu_n)<0$ for all $n\in\N$.

To construct $(\mu_n)_{n\in \N}$ we start by taking a minimizing sequence $(\rho_n)_{n\in\N}\subset\mathcal{P}(\R^d)\cap L^m(\R^d)$, i.e., a sequence satisfying $\lim_{n\to\infty}E_\eps(\rho_n)=\inf E_\eps$. Because of \eqref{eq:rearrangement-energy}, without loss of generality we can take $\rho_n$ to be radially symmetric and decreasing for all $n\in\N$. For each $n$, we construct $\mu_n$ from $\rho_n$ by optimizing over dilations. To this end, we fix $n\in\N$ and we can explicitly compute
\begin{equation*}
E_\eps(T_r\#\rho_n)=r^\beta \mathcal{W}_\beta(\rho_n)+\eps r^{d(1-m)}\mathcal{E}_m(\rho^n),
\end{equation*}
where we remind that $T_r(x)=rx$. Using that $\mathcal{W}_\beta(\rho^n)<0$ and $\mathcal{E}_m(\rho^n)>0$ combined with the hypothesis that $d(1-m)<\beta<0$, we get that there exists a unique $r_{0,n}\in (0,\infty)$ such that $E_\eps(T_r\#\rho_n)$ is strictly decreasing for $r<r_{0,n}$ and strictly increasing for $r>r_{0,n}$, and $E_\eps(T_{r_{0,n}}\#\rho_n)<0$. Furthermore, we have
\begin{equation*}
0=\left.\frac{dE_\eps(T_r\#\rho_n)}{dr}\right|_{r=r_{0,n}}=\beta (r_{0,n})^{\beta-1} \mathcal{W}_\beta(\rho_n)+\eps d(1-m)(r_{0,n})^{d(1-m)-1}\mathcal{E}_m(\rho_n).
\end{equation*}
By multiplying by $r_{0,n}$ and undoing the scaling, we get the equation
\begin{equation*}
\mathcal{W}_\beta (T_{r_{0,n}}\# \rho_n)= \eps \frac{d(m-1)}{\beta}\mathcal{E}_m(T_{r_{0,n}}\# \rho_n).
\end{equation*}
Now, since we have chosen $r_{0,n}$ to minimize over dilations we also have
\begin{equation*}
E_\eps(T_{r_{0,n}}\# \rho_n)\le E_\eps(\rho_n).
\end{equation*}
Hence, we can take $\mu_n=T_{r_{0,n}}\# \rho_n$ that satisfies \eqref{propertiesofmun} and is radially decreasing because $\rho_n$ is radially decreasing. To conclude, we can write
\begin{equation}\label{relation}
\inf E_\eps \leq E_\eps(\mu_n)=\eps \frac{d(m-1)+\beta}{\beta}\mathcal{E}_m(\mu_n)<0
\end{equation}
since $\mu_n\in\mathcal{P}(\R^d)\cap L^m(\R^d)$ and $0>\beta>(1-m)d$.

\medskip

\textit{Step 2.} By Theorem \ref{weakcompact}, we know that there exists $\mu_\infty\in \mathcal{M}_+(\R^d)$ and a subsequence $\{\mu_{n_i}\}_{i\in \N}$, such that $\mu_{n_i}\rightharpoonup\mu_\infty$ weak-* as measures and $0\le\mu_\infty(\R^d)\le 1$. Abusing notation, we do not keep track of this subsequence and we still denote by $E_\eps(\mu_\infty)$ the energy functional evaluated at $\mu_\infty$ even if we do not know yet if it is a probability measure. In this step, we show that
\begin{equation}\label{liminfformun}
\inf E_\eps=\liminf_{n\to\infty}E_\eps(\mu_n)\ge E_\eps(\mu_\infty).
\end{equation}
We recall that $E_\eps(\mu_n)=\mathcal{W}_\beta(\mu_n)+\eps \mathcal{E}_m(\mu_n)$. Exploiting the lower semicontinuity of the entropy functional we have
\begin{equation*}
\liminf_{n\to\infty}\mathcal{E}_m(\mu_n)\ge \mathcal{E}_m(\mu_\infty).
\end{equation*}
We show \eqref{liminfformun} by proving that
\begin{equation}\label{limit for the interaction}
\lim_{n\to\infty}\mathcal{W}_\beta(\mu_n)=\mathcal{W}_\beta(\mu_\infty).
\end{equation}
Using the properties of $\mu_n$ in \eqref{propertiesofmun}, we get
\begin{equation*}
\ird\ird |x-y|^\beta\;d\mu_n(x)d\mu_n(y)=\eps d \ird \mu_n^m(x)\;dx.
\end{equation*}
Using that $\beta>\max\{-d,d(1-m)\}$, we pick $\gamma=(\beta+\max\{-d,d(1-m)\})/2$. By the variation of the Hardy--Littlewood--Sobolev inequality Theorem~\ref{variationofHLS} and Young's inequality, we know that there exists a constant  $C(d,m,\beta)>0$ that depends on dimension, $\beta$ and $m$ such that
\begin{align}\label{eq:applyingHLS}
\ird\ird |x-y|^\beta\;d\mu_n(x)d\mu_n(y)&=\eps d\ird \mu_n^m(x)\;dx \nonumber \\ 
& \ge \eps C(d,m,\beta)\ird\ird|x-y|^{\gamma}\;d\mu_n(x)d\mu_n(y)-\eps d.
\end{align}
Because $\gamma<\beta$, we know that there exists $\eta(d,m,\beta,\eps)>0$, such that $|z|^\beta\le \frac{1}{2}\eps C(d,m,\beta) |z|^{\gamma}$ for every $z\in B_\eta$. Therefore, we deduce
\begin{equation*}
\int\int_{|x-y|<\eta}|x-y|^\beta\;d\mu_n(x)d\mu_n(y)\le \frac{1}{2}\eps C(d,m,\beta)\int\int_{|x-y|<\eta} |x-y|^{\gamma}\;d\mu_n(x)d\mu_n(y).
\end{equation*}
Updating \eqref{eq:applyingHLS}, we get that for every $n\in\N$
\begin{equation*}
\eta^\beta+\eps d\ge\int\int_{|x-y|>\eta}|x-y|^\beta\;d\mu_n(x)d\mu_n(y)+\eps d\ge \frac{1}{2}\eps C(d,m,\beta)\ird\ird|x-y|^{\gamma}\;d\mu_n(x)d\mu_n(y).
\end{equation*}
Hence, updating the constants we obtain that there exists $C(d,m,\beta,\eps)>0$, such that
\begin{equation}\label{boundonthegammainteraction}
\ird\ird|x-y|^{\gamma}\;d\mu_n(x)d\mu_n(y)\le C(d,m,\beta,\eps)\qquad\mbox{for all $n$.}
\end{equation}

Next we use \eqref{boundonthegammainteraction} to pass to the limit in the interaction energy \eqref{limit for the interaction}. We first realize that by lower semicontinuity we have the inequality
\begin{equation}\label{liminf for the interaction}
\liminf_{n\to\infty} \ird\ird|x-y|^\beta d\mu_n(x)d\mu_n(y)\ge \ird\ird|x-y|^\beta d\mu_\infty(x)d\mu_\infty(y).
\end{equation}
Hence, we can show \eqref{limit for the interaction} by showing the reverse inequality to \eqref{liminf for the interaction}. 

For any $\delta>0$ and $R>0$, we bound the interaction energy by
\begin{equation}\label{eq:decomposition}
\begin{split}
\ird\ird |x-y|^\beta \;d\mu_n(x)d\mu_n(y)\le &\displaystyle\int\int_{\{|x-y|<\delta\}}|x-y|^\beta \;d\mu_n(x)d\mu_n(y)\\ 	&+\int\int_{\{\delta^{-1}>|x-y|>\delta\}\cap \{|x|<R\}\cap \{|y|<R\} }|x-y|^\beta \;d\mu_n(x)d\mu_n(y)\\
	&+2\int\int_{\{\delta^{-1}>|x-y|>\delta\}\cap \{|x|>R\}}|x-y|^\beta \;d\mu_n(x)d\mu_n(y)\\
	&+\int\int_{\{\delta^{-1}<|x-y|\}}|x-y|^\beta \;d\mu_n(x)d\mu_n(y).
\end{split}
\end{equation}
Analyzing the first term on the right hand side, we get
\begin{equation}\label{eq:firsterm}
\begin{split}
\displaystyle\int\int_{\{|x-y|<\delta\}}|x-y|^\beta \;d\mu_n(x)d\mu_n(y)
&\displaystyle\le \int\int_{\{|x-y|<\delta\}}|x-y|^{\beta-\gamma} |x-y|^{\gamma}\;d\mu_n(x)d\mu_n(y)\\
&\displaystyle\le \delta ^{\beta-\gamma}\int\int_{\{|x-y|<\delta\}}|x-y|^{\gamma}\;d\mu_n(x)d\mu_n(y)\\
&\le \delta ^{\beta-\gamma}C(d,m,\beta,\eps),
\end{split}
\end{equation}
where in the last bound we have used \eqref{boundonthegammainteraction}.
The second term has no singularity, hence we can pass to the limit
\begin{equation}\label{eq:secondterm}
\limsup_{n\to\infty}\int\int_{\{\delta^{-1}>|x-y|>\delta\}\cap \{|x|<R\}\cap \{|y|<R\}} \!\!\!\!\!\!\!\!\!\!\!\!\!\!\!\!\!\!|x-y|^\beta \;d\mu_n(x)d\mu_n(y)\le\ird\ird |x-y|^\beta d\mu_\infty(x)d\mu_\infty(y). 
\end{equation}
For the third term, we get
\begin{align}\label{eq:thirdterm}
	2\ird\int_{\{x\in\R^d \mid \delta^{-1}>|x-y|>\delta,\; |x|>R\}} \!\!\!\!\!\!\!\!\!\!\!\!\!\!\!\!\!\!|x-y|^\beta \;d\mu_n(x)d\mu_n(y)
	&\le 2 \|\mu_n\|_{L^\infty(B_R^\mathrm{c})} \delta^\beta \ird \int_{\{x\in\R^d \mid \delta^{-1}>|x-y|>\delta,\; |x|>R\}} \!\!\!\!\!\!\!\!\!\!\!\!\!\!\!\!\!\! d x \,d\mu_n(y)\nonumber\\
	&\leq 2 \|\mu_n\|_{L^\infty(B_R^\mathrm{c})} \delta^\beta \omega_d \delta^{-d} \le 2\delta^{\beta-d} R^{-d},
\end{align}
where we have used that $\beta<0$ and that because $\mu_n$ has unit mass and is radially decreasing we have the inequality $\mu_n(x)\le (\omega_d |x|^d)^{-1}$. Analyzing the fourth term, we get
\begin{equation}\label{eq:fourthterm}
	\int\int_{\{\delta^{-1}<|x-y|\}}|x-y|^\beta \;d\mu_n(x)d\mu_n(y)\le \delta^{-\beta}.
\end{equation}
Putting \eqref{eq:decomposition}, \eqref{eq:firsterm}, \eqref{eq:secondterm}, \eqref{eq:thirdterm} and \eqref{eq:fourthterm} together, we obtain that, for any $\delta>0$ and $R>0$,
\begin{equation*}
	\limsup_{n\to\infty}\ird\ird |x-y|^\beta \,d\mu_n(x)d\mu_n(y)\!\le\! C\delta ^{\beta-\gamma}\!+\!\ird\ird |x-y|^\beta \,d\mu_{\infty}(x)\,d\mu_\infty(y)\!+\!2\delta^{\beta-d} R^{-d}\!+\!\delta^{-\beta}.
\end{equation*}
First taking $R\to\infty$ and then $\delta\to 0$ we recover the reverse inequality to \eqref{liminf for the interaction}. This shows the desired property \eqref{liminfformun}.

\medskip

\textit{Step 3.} In this step we show by contradiction that $\mu_\infty(\R^d)=1$. 

We first notice that $\mu_\infty(\R^d)>0$. In fact, if $\mu_\infty(\R^d)=0$, then $\mu_\infty=0$ which implies by \eqref{liminfformun}
\begin{equation*}
0>\inf E_\eps\ge E_\eps(\mu_\infty)=0
\end{equation*}
a contradiction. The fact that $\inf E_\eps<0$ follows from \eqref{relation}. 

Next, we derive a contradiction if $1>\mu_\infty(\R^d)>0$. From $\mu_\infty$ we construct $\tilde{\mu}_\infty\in\mathcal{P}(\R^d)$, such that by \eqref{liminfformun}
\begin{equation}\label{contradictionofmuinfty}
\inf E_\eps\ge E_\eps(\mu_\infty)>E_\eps(\tilde{\mu}_\infty)\ge \inf E_\eps,
\end{equation}
the desired contradiction. We construct $\tilde{\mu}_\infty$ by an appropriate scaling of $\mu_\infty$. We define
$$
\tilde{\mu}_\infty=r_0^{d} \, T_{r_0}\#\mu_\infty,
\qquad \mbox{where} \qquad
r_0=(\mu_\infty(\R^d))^{-1/d}
$$
is chosen to satisfy the mass condition.

Next we analyze the behavior of the different parts of our energy under this scaling. Looking at the entropy, we get
\begin{equation}\label{behaviorofentropy}
\mathcal{E}_m(r_0^{d}T_{r_0}\#\mu_\infty)=r_0^{md}\mathcal{E}_m(T_{r_0}\#\mu_\infty)=r_0^{md}r_0^{(1-m)d}\mathcal{E}_m(\mu_\infty)=r_0^{d}\mathcal{E}_m(\mu_\infty).
\end{equation}
For the interaction energy, unpacking the scaling we get
\begin{equation}\label{behaviouroftheinteractionenergy}
\mathcal{W}_\beta(r_0^{d} T_{r_0}\#\mu_\infty)=r_0^{2d}\mathcal{W}_\beta(T_{r_0}\#\mu_\infty)=r_0^{2d}r_0^\beta\mathcal{W}_\beta(\mu_\infty)=r_0^{2d+\beta}\mathcal{W}_\beta(\mu_\infty).
\end{equation}
By the hypothesis that $\beta>-d$ and that $r_0>1$, we get that $r_0^{2d+\beta}>r_0^d$. Using \eqref{behaviorofentropy} and \eqref{behaviouroftheinteractionenergy} by looking at the full energy we get
\begin{equation*}
E_\eps(\tilde{\mu}_\infty)=r_0^{2d+\beta}\mathcal{W}_\beta(\mu_\infty)+r_0^{d}\mathcal{E}_m(\mu_\infty)<r_0^d E_\eps(\mu_\infty)\le r_0^d\inf_{\mathcal{P}(\R^d)} E_\eps<\inf_{\mathcal{P}(\R^d)} E_\eps,
\end{equation*}
which shows the desired contradiction \eqref{contradictionofmuinfty} and finishes the proof of the Theorem.
\end{proof}

\begin{remark}
The previous is result is a particular case of \cite[Theorem II.1]{lions1984concentration}; see also \cite{CHMV} for more comments. As a main difference, our proof exploits the extra rigidity stemming from the radially decreasing rearrangement, bypassing the need of exploiting the subadditivity of the energy of minimizers at different mass. Moreover, we make use of the HLS inequality to obtain lower semicontinuity of the energy for the minimizing sequence.
\end{remark}


\section{Sharp condition on the existence of minimizers for $m=1$}
In the case of linear diffusion, we can show that Theorem~\ref{thm:nonexistence-global} is sharp for general interaction potentials. We should also note that this case is not considered in \cite{lions1984concentration}.
\begin{theorem}[sharp condition on existence of global minimizers for linear diffusion] \label{thm:existence-global}
Suppose that the entropy function is given by $U(r)= r\log r$ and that $W$ is positive and satisfying {\bf (H)}. 
If 
\begin{equation}\label{hyp: smaller than the log}
	 \limsup_{|x|\to\infty} \nabla W(x) \cdot x <2d\eps,
\end{equation}
then $E_\eps$ is not bounded below.
Alternatively, if
\begin{equation}\label{hyp: bigger than the log}
	 \liminf_{|x|\to\infty} \nabla W(x) \cdot x >2d\eps,
\end{equation}
then there exists $\rho_\infty\in\mathcal{P}(\R^d)$ such that
\begin{equation*}
E_\eps(\rho_\infty)=\inf E_\eps>-\infty.
\end{equation*}
\end{theorem}

\begin{proof}
If we are under the hypothesis \eqref{hyp: smaller than the log}, then we can check that we satisfy hypothesis \eqref{hyp:nonexistence} of Theorem~\ref{thm:nonexistence-global}; hence, the energy is not bounded below. In fact, when $U(\rho)=\rho\log\rho$, we have $v(r)=d$ and is independent of $r\in(0,\infty)$ (as mentioned in Section \ref{subsec:defn-energy}). Thus, \eqref{hyp:nonexistence} is exactly
\begin{equation*}
 \limsup_{|x|\to\infty}\frac{1}{2} \nabla W(x) \cdot x -\eps d<0,
\end{equation*}
which is equivalent to \eqref{hyp: smaller than the log}. If now we are under the hypothesis \eqref{hyp: bigger than the log}, we show below in three steps by means of the logarithmic HLS inequality that the energy is bounded from below and that a minimizer exists.

\medskip

\textit{Step 1.} We show that under hypothesis \eqref{hyp: bigger than the log}, there exists $\delta\in(0,1/2)$ and $L\in \R$, such that
\begin{equation}\label{eq:Wboundedbylog}
W(x)\ge \frac{2d\eps}{1-\delta}\log|x|+L.
\end{equation}
From \eqref{hyp: bigger than the log} we know that there exists $\delta\in (0,1/2)$ such that 
\begin{equation*}
	 \liminf_{|x|\to\infty} \nabla W(x) \cdot x >\frac{2d\eps}{1-2\delta}.
\end{equation*}
Therefore, we can say that there exists $R_0$ such that
\begin{equation*}
\nabla W(x)\cdot x\ge \frac{2d\eps}{1-\delta} \qquad\mbox{for all $|x|>R_0$.}
\end{equation*}
We define $L$ by
\begin{equation*}
L=\inf_{z\in \overline{B_{R_0}}} W(z)-\frac{2d\eps}{1-\delta} \log(|z|).
\end{equation*}
Taking $x\in B_{R_0}^\mathrm{c}$, we consider the function $g(t)=W((1-t)x_0+tx)-\frac{2d\eps}{1-\delta}\log(|(1-t)x_0+tx|)$, $t\in[0,1]$, where $x_0=R_0\frac{x}{|x|}$. Next, we notice that $g$ is increasing. In fact, using that $x$ and $x_0$ are parallel and that $(1-t)x_0+tx\in B_R^\mathrm{c}$ we can see, for all $t\in[0,1]$,
\begin{equation*}
\begin{split}
\displaystyle g'(t)&\displaystyle=\nabla W((1-t)x_0+tx)\cdot(x-x_0)-\frac{2d\eps}{1-\delta}\frac{(1-t)x_0+tx}{|(1-t)x_0+tx|^2}\cdot (x-x_0)\\
&\displaystyle=\frac{|x|-R_0}{(1-t)R_0+t|x|}\left(\nabla W((1-t)x_0+tx)\cdot ((1-t)x_0+tx)- \frac{2d\eps}{1-\delta}\right)>0.
\end{split}
\end{equation*}
Therefore, using the definition of $g$, its monotonicity and the definition of $L$ we have
\begin{equation*}
W(x)-\frac{2d\eps}{1-\delta}\log|x|=g(1)\ge g(0)=W\left(R_0\frac{x}{|x|}\right)-\frac{2d\eps}{1-\delta}\log(R_0)\ge L,
\end{equation*}
which shows \eqref{eq:Wboundedbylog}.

\medskip

\textit{Step 2.} Using the behavior of $W$ \eqref{eq:Wboundedbylog} and the logarithmic HLS inequality, we show that 
\begin{equation}\label{lowerboundedness}
	\inf \left( (1-\delta)\mathcal{W}+\eps \mathcal{E} \right)>-\infty,
\end{equation}
where $\delta$ is given by Step 1.

By the logarithmic HLS inequality (Theorem~\ref{thm: log HLS}) we have that there exists $C\in \R$, such that for any $\rho\in C_\mathrm{c}^\infty(\R^d)\cap \mathcal{P}(\R^d)$
\begin{equation}\label{logHLS}
-\ird\ird \log(|x-y|)\; d\rho(x)d\rho(y)\le \frac{1}{d}\mathcal{E}(\rho)+C.
\end{equation}
By Step 1, we notice that there exists $L\in \R$ such that
\begin{equation*}
\frac{(1-\delta)}{2}\ird\ird W(|x-y|)\, d\rho(x)d\rho(y)\ge d\eps \ird\ird \log(|x-y|)\, d\rho(x)\,d\rho(y)+\frac{(1-\delta)L}{2}.
\end{equation*}
Combining this bound with \eqref{logHLS}, we get that for any $\rho\in C_\mathrm{c}^\infty(\R^d)\cap \mathcal{P}(\R^d)$
\begin{equation*}
(1-\delta)\mathcal{W}(\rho)+\eps \mathcal{E}(\rho)\ge \eps \mathcal{E}(\rho)+ d\eps \ird\ird \log(|x-y|)\, d\rho(x)\,d\rho(y)+\frac{(1-\delta)L}{2}\ge -dC+\frac{(1-\delta)L}{2}>-\infty.
\end{equation*}
We define $C_\eps=-dC+\frac{L}{4}$. By the definition of the energy, we have that given any $\mu\in \mathcal{P}(\R^d)$
\begin{equation*}
(1-\delta)\mathcal{W}(\mu)+\eps \mathcal{E}(\mu)=
\sup_{\substack{ \{\rho_n\}\subset C_\mathrm{c}^\infty(\R^d)\\ \mathrm{s.t.}\; \rho_n\rightharpoonup\mu}}
\liminf_{n\to\infty} (1-\delta)\mathcal{W}(\rho_n)+\eps \mathcal{E}(\rho_n)\ge C_\eps,
\end{equation*}
which shows \eqref{lowerboundedness}.

\medskip

\textit{Step 3.} We show now that any minimizing sequence $\{\rho_k\}_{k\in\N}\subset\mathcal{P}(\R^d)$ is tight up to translations, and hence, by lower semicontinuity there exists $\rho_\infty\in\mathcal{P}(\R^d)$ such that
\begin{equation*}
E_\eps(\rho_\infty)=\inf E_\eps >-\infty.
\end{equation*}

\medskip

By Step 2, we know that there exists $\delta>0$ and $C_\eps\in \R$, such that
\begin{equation*}
E_\eps(\rho_k)\ge \delta\mathcal{W}(\rho_k)+C_\eps.
\end{equation*}
Finally, we apply Lemma~\ref{lem tightness through the interaction energy} combined with the observation that
\begin{equation*}
\liminf_{k\to\infty} \mathcal{W}(\rho_k)\le \frac{1}{\delta}(\liminf_{k\to\infty} E_\eps(\rho_k)-C_\eps)<\infty
\end{equation*}
to get that there exists $\rho_\infty\in\mathcal{P}(\R^d)$ such that, up to translations and a subsequence,
\begin{equation*}
\rho_k\rightharpoonup\rho_\infty \quad \mbox{weakly-$^*$ as $k\to\infty$}.
\end{equation*}
Finally, the fact that $\rho_\infty$ is a minimizer follows from the lower semicontinuity of the energy.
\end{proof}

\begin{remark}\label{connectiontoKS}
Theorem \ref{thm:existence-global} can be further generalized to allow for certain singularity of the interaction potential at the origin. More precisely, one can follow the same proof to show that if the interaction potential satisfies
\begin{equation*}
	 \liminf_{|x|\to\infty} \left( W(x) -\frac{\eps}{2d} \log |x| \right) =\infty
\end{equation*}
and
$$
\inf_{x\in\R^d} \left( W(x) -\frac{\eps}{2d} \log |x| \right) > -\infty,
$$
then $\inf E_\eps>-\infty$ and minimizing sequences are tight. Further arguments are needed to show that the infimum is achieved, see \cite{BCC08,CCV} for related arguments.
\end{remark}

\begin{remark}\label{connectiontoKS}
For the energy functional 
\begin{equation*}
	\frac{1}{2}\ird\ird \log |x-y|\,d\rho(x)\,d\rho(y) + \eps \ird \rho(x) \log \rho(x)\,dx, 
\end{equation*}
corresponding to the classical Keller--Segel model \cite{DoPe04,BlaDoPe06,BCC12}, it is known that there is a critical value of the noise, $\eps_\mathrm{c}=1/(2d)$, such that the energy functional is bounded from below if and only if $\eps=\eps_\mathrm{c}$. Moreover, the optimizers of the logarithmic HLS \eqref{logHLS} are equivalent to the set of stationary states for this critical value $\eps_\mathrm{c}$. Similarly, our previous theorem shows that if $W$ is bounded from below and
$$
 \lim_{|x|\to\infty} \nabla W(x) \cdot x =L>0\,,
$$
then there also exists a critical diffusion $\eps_\mathrm{c}=L/(2d)$ separating the existence of steady states from the unboundeness from below of the free energy. Notice that these two hypotheses on $W$ allow us to show for $0<\eps<\eps_\mathrm{c}$ that the energy is bounded below and that there is confinement for minimizing sequences, respectively. 
\end{remark}


\bibliographystyle{abbrv}
\bibliography{bibliography}
\end{document}